\documentclass[12pt,reqno]{amsart}
\usepackage{amssymb}
\usepackage{mathrsfs}
\usepackage[all]{xy}

\theoremstyle{plain}
\newtheorem{prop}{Proposition}

\newtheorem{theo}[prop]{Theorem}

\theoremstyle{remark}
\newtheorem{rema}[prop]{Remark}
\theoremstyle{definition}
\newtheorem{defi}[prop]{Definition}

\numberwithin{equation}{section}

\newcommand{\PP}{{\mathbb P}}
\newcommand{\bP}{{\mathbb P}}
\newcommand{\Q}{{\mathbb Q}}
\newcommand{\F}{{\mathbb F}}
\newcommand{\G}{{\mathbb G}}

\newcommand{\Z}{{\mathbb Z}}

\newcommand{\cO}{{\mathcal O}}
\newcommand{\cI}{{\mathcal I}}

\newcommand{\ra}{\rightarrow}

\newcommand{\eqto}{\stackrel{\lower1.5pt\hbox{$\scriptstyle\sim\,$}}\to}

\DeclareMathOperator{\Pic}{Pic}
\DeclareMathOperator{\Spec}{Spec}

\DeclareMathOperator{\Br}{Br}
\DeclareMathOperator{\Aut}{Aut}

\DeclareMathOperator{\coker}{coker}
\DeclareMathOperator{\Nm}{Nm}

\DeclareMathOperator{\cores}{cores}

\begin{document}
\title[Fibrations in Sextic del Pezzo surfaces]{Fibrations in sextic del Pezzo surfaces with mild singularities}
\author{Andrew Kresch}
\address{
  Institut f\"ur Mathematik,
  Universit\"at Z\"urich,
  Winterthurerstrasse 190,
  CH-8057 Z\"urich, Switzerland
}
\email{andrew.kresch@math.uzh.ch}
\author{Yuri Tschinkel}
\address{
  Courant Institute,
  251 Mercer Street,
  New York, NY 10012, USA
}
\address{Simons Foundation\\
160 Fifth Avenue\\
New York, NY 10010\\
USA}
\email{tschinkel@cims.nyu.edu}

\date{June 11, 2019} 

\maketitle

\section{Introduction}
\label{sec.introduction}
In this paper we study families of degree $6$ del Pezzo surfaces over
higher-dimensional bases,
with a view toward existence of good models.
There is an extensive literature on degenerations of del Pezzo surfaces,
see, e.g., \cite{fujita}, \cite{corti}, and
specifically sextic del Pezzo surfaces, e.g.,
\cite{fukuokarelative}.
These degenerations play a role in the
Minimal Model Program
and enter the study of related moduli problems, e.g.,
\cite{schaffler}, \cite{fukuokarefinement}.
They are also relevant in arithmetic applications and the study of
rationality, see \cite{AHTV}.

Our starting point is the work of
Blunk \cite{blunk} and Kuznetsov \cite{kuznetsovsextic},
describing isomorphism classes of sextic del Pezzo surfaces in terms
of supplementary data, taking the form of \'etale algebras
over a given base field of degrees $2$ and $3$,
together with Brauer group elements.

This paper continues our previous work, in which we studied
families of del Pezzo surfaces of degree
$9$ (Brauer-Severi surfaces \cite{KTsurf}, \cite{KTsb})
and $8$ (involution surfaces \cite{KTinvsurf}, \cite{KTbrinv},
which include the case of quadric surfaces).
(The case of del Pezzo surfaces of
degree $7$ is not interesting, since they are never minimal.)
The passage to families entails a study of ramification patterns for the
\'etale algebras and the Brauer group elements.

In our previous work, as well as here,
the main results concern the existence of families,
with specified fiber types and global singularity descriptions.
This is applied, for instance, to Brauer group computations
and constructions in families, e.g., in
the study of rationality problems as in
\cite{HKTconic}, \cite{HKTthree}.

In this paper, we require a
\emph{regular branch divisor} for the covers and allow only
limited ramification of Brauer group elements.
From a geometric point of view, we need to specify
coverings of the base of degrees $2$ and $3$ with
Brauer classes satisfying compatibility conditions.
We obtain five degeneration types, corresponding to the possible ramification
of the pair of coverings,
which we call \emph{basic}.
These are parametrized by
conjugacy classes of cyclic subgroups of $\mathfrak S_2\times \mathfrak S_3$;
see Section \ref{sect.basicdegen}.
While the structure of degree $2$ covers is well understood,
the branch locus of degree $3$ covers is typically singular and becomes
nonsingular only after birational modification of the base;
see \cite{KTtriple} and references therein.
In this paper we do not address
codimension $2$ phenomena, e.g.,
when the branch loci of the covers intersect.
Similar restrictions were present in the work of
Auel, Parimala, and Suresh on quadric surface bundles
\cite{APS}.

To draw a comparison with involution surface bundles,
basic degeneration corresponds to
Type I in \cite{KTinvsurf}, while the more general notion from op.\ cit.\ of mild degeneration
(essentially, degenerations that occur in families over a regular base
with singular fibers over a regular divisor) encompasses three additional types that involve ramification
of the Brauer group element.
In the case of Brauer-Severi surface bundles \cite{KTsurf},
there is only one kind of degeneration that occurs over a regular divisor,
connected with ramification of the Brauer group element.

\begin{defi}
Let $S$ be a regular scheme.
A \emph{sextic del Pezzo surface bundle}, or \emph{DP6 bundle}, over $S$ is a
flat projective morphism $\pi\colon X\to S$ such that the
locus $U\subset S$ over which $\pi$ is smooth is dense in $S$ and
the fibers of $\pi$ over points of $U$ are del Pezzo surfaces of degree $6$.
\end{defi}

Now we suppose that $2$ and $3$ are invertible in the local rings of $S$.

\begin{defi}
A sextic del Pezzo surface bundle $\pi\colon X\to S$
has \emph{basic degeneration} if
every singular fiber is geometrically isomorphic to one of the
following schemes:
\begin{itemize}
\item (Type I) four copies of the blow-up of $3$ collinear points on $\PP^2$ with
the $(-2)$-curve contracted, with each copy glued to the other three along the three
exceptional curves;
\item (Type II) the blow-up of two general points on a
quadric surface with $\mathsf A_1$-singularity;
\item (Type III) ``pinched'' $\F_2$, obtained by gluing a
Hirzebruch surface $\F_2$ to itself along a degree $2$ covering
$(\text{$(-2)$-curve})\cong \PP^1\to \PP^1$;
\item (Type IV) three copies $L_i$, $i\in \Z/3\Z$, of the
blow-up of two general points on a fiber of $\F_2$, with the
proper transform of the fiber collapsed to an $\mathsf A_1$-singular point
and images of exceptional divisors denoted by
$\ell_i$, $\ell'_i$, and of the $(-2)$-curve, by $\lambda_i$,
together with three quadric surfaces $M_i$, $i\in \Z/3\Z$, with
rulings $m_i$ and $m'_i$ meeting at a point on a smooth conic $\mu_i$,
each with multiplicity $2$,
with following pairs of curves glued:
\[ (\lambda_i,\mu_i),\,\, (\ell_i,m_{i+1}),\,\, (\ell'_i,m'_{i+2}),
\qquad\forall\,i\in \Z/3\Z; \]
\item (Type V) Hirzebruch surface $\F_4$ with fiber and
$(-4)$-curve glued.
\end{itemize}
\end{defi}

For the gluing of a scheme to itself along a finite morphism from a closed subscheme to
another scheme, see \cite{ferrand}.
In particular, the fibers of Types I, III, IV, and V are not normal,
and Type IV fibers are not even reduced.

\begin{rema}
Du Val degenerations of del Pezzo surfaces have been studied
classically, see, e.g., \cite{coraytsfasman};
above, only Type II is of this form.
Degenerations of sextic del Pezzo surfaces
with special fiber that is irreducible but not normal
are considered in \cite[Table 3]{fukuokarelative};
our Types III and V are listed there.
Types I and IV have not appeared in the literature; these surfaces are
only embedded by a nontrivial multiple of the anticanonical class.
\end{rema}

Our approach to the construction of families is a systematic application
of root stacks and descent, as in \cite{HKTconic}, \cite{KTsurf}, \cite{KTinvsurf}.
In essence, we exchange ramification of the covers for extra stacky structure.
Other approaches have been used; for instance,
the Minimal Model Program has been applied by Corti \cite{corti} to produce
degenerations over a DVR, with normal irreducible special fiber,
but only under the assumption of an algebraically closed residue field.

The families with basic degeneration that we will construct
(Theorems \ref{thm.construction} and \ref{thm.blunknotsosmooth}) will have explicit global singularity
descriptions.
In particular, we observe new phenomena:
\begin{itemize}
\item degenerations with fibers embedded by
multiples of the anticanonical class, according to
Gorenstein index of $\Q$-Gorenstein singularities of the total space
(Type I);
\item singularities that are not $\Q$-Gorenstein
(Type IV), requiring adjustment by components of the fibers to turn
(an appropriate multiple of) the anticanonical class into a Cartier divisor class.
\end{itemize}

In Section \ref{sect.classify}, we recall the supplementary data attached to
the classification of sextic del Pezzo surfaces over fields and
exhibit smooth families of sextic del Pezzo surfaces, as models in the
case of a general base with \emph{unramified} covers and Brauer group elements
(Theorem \ref{thm.blunksmooth}).
In Section \ref{sect.basicdegen}, we analyze the
five degeneration types.
Section \ref{sect.construction} carries out the construction and
establishes the main theorems
(Theorem \ref{thm.construction}, with the basic construction,
and Theorem \ref{thm.blunknotsosmooth}, exhibiting the desired families).

\medskip
\noindent
\textbf{Acknolwedgments:}
The authors are grateful to Asher Auel and Brendan Hassett for helpful discussions.
The second author was partially supported by
NSF grant 1601912.

\section{Smooth families of sextic del Pezzo surfaces}
\label{sect.classify}
Let $K$ be a field. 
A smooth del Pezzo surface $X$ over $K$ of degree $\mathsf d$ is a smooth projective surface with 
ample anticanonical class $\omega^{-1}_X$, and 
\[
\mathsf d=\mathsf d(X):=\omega_X^2.
\]
Here $\mathsf d$ takes the values $1,\ldots, 9$.
When $K$ is algebraically closed, we have the following description:
when $\mathsf d=9$, $X\cong\bP^2$;
when $\mathsf d=8$, $X$ is isomorphic to $\bP^1\times \bP^1$ or the blowup of $\bP^2$ in one point;
when $\mathsf d\le 7$, $X$ may be obtained by blowing up $\bP^2$ in 
$9- \mathsf d$ points in general position. 
Smooth irreducible curves $C$ on $X$ with $C^2= -1$ are called \emph{exceptional curves},
and the union of the exceptional curves is the \emph{exceptional locus}.
In this paper, we focus on the case $\mathsf d=6$, in which case the
geometric automorphism group of $X$ fits into a split exact sequence
\begin{equation}
\label{eqn:aut}
1\ \ra T\ra \Aut(\overline{X})\ra \mathfrak  S_2\times \mathfrak S_3\ra 1.
\end{equation}

\begin{prop}[Blunk construction \cite{blunk}]
A degree 6 del Pezzo surface over a field $K$ is classified by 
\begin{itemize}
\item[--] \'etale $K$-algebras $L$ and $M$ of respective degrees $2$ and $3$,
\item[--] Brauer group elements $\eta\in \Br(L)[3]$ and $\zeta\in \Br(M)[2]$,
each restricting to $0$ in $\Br(L\otimes_KM)$
and corestricting to $0$ in $\Br(K)$.
\end{itemize}
\end{prop}

Blunk's construction builds on work of
Colliot-Th\'el\`ene, Karpenko, and Merkurjev
\cite{CTKM}.
A table listing possibilities for the Blunk data, with
corresponding arithmetic invariants, is given in
\cite[Table 4]{auelbernardara}.

The following are the essential ingredients to the Blunk construction.
\begin{itemize}
\item A two-dimensional torus $T$ is canonically associated
with $L$ and $M$, together with a
toric variety $X_0$ for $T$ that is a
sextic del Pezzo surface.
The parameter scheme of
unordered triples of pairwise disjoint exceptional curves of $X_0$ is
$\Spec(L)$ and of unordered pairs of opposite exceptional curves is
$\Spec(M)$.
\item Sextic del Pezzo surfaces for $T$
(i.e., with $T$ as identity component of the
automorphism group scheme) correspond to torsors under $T$, i.e.,
elements of $H^1(K,T)$.
\item Exact sequences lead to a description of $H^1(K,T)$ in terms of
pairs of Brauer group elements satisfying the stated conditions.
\end{itemize}
The inverse map $T\to T$ extends to an automorphism of $X_0$.
We call a pair of points, or curves, of $X_0$ \emph{opposite} if they are exchanged under
this automorphism.
The same terminology applies to exceptional curves of $X$, and to the
singular points of the exceptional locus of $X$.

We describe $X$ as \emph{rigidified} by $L/K$ and $M/K$ when
the parameter scheme of pairwise disjoint exceptional curves of $X$,
respectively of unordered pairs of opposite exceptional curves,
is identified with $\Spec(L)$, respectively $\Spec(M)$.
The same terminology will be used for $\pi\colon X\to S$,
for any integral scheme $S$ with rational function field $K$.
For instance, $\pi\colon X\to S$ could be a
smooth DP6 bundle, i.e., a DP6 bundle such that $\pi$ is smooth.
We also allow $S$ to be a disjoint union of finitely many
integral schemes, in which case the rigidification data consist of \'etale algebras
over $K:=K_1\times\dots\times K_n$, where $K_1$, $\dots$, $K_n$ are
the residue fields at the generic points of the components of $S$.
Such \'etale algebras will be called \emph{rigidification data}.

Blunk's description of sextic del Pezzo surfaces over a field
is best expressed as a description of isomorphism classes of
sextic del Pezzo surfaces with rigidification.
A sextic del Pezzo surface rigidified by $L$ and $M$ is determined uniquely
up to isomorphism by $\eta\in \Br(L)[3]$ and $\zeta\in \Br(M)[2]$ satisfying
the indicated conditions.

We seek a generalization to an arbitrary regular base.
For this, we need to consider
smooth families of sextic del Pezzo surfaces with rigidification as
\emph{equivalent} when there exists a birational equivalence between
them that restricts over the generic point (of every component of the base)
to an isomorphism of rigidified sextic del Pezzo surfaces.

\begin{theo}
\label{thm.blunksmooth}
Let $S$ be a quasi-compact separated regular scheme
with rigidification data $L/K$, $M/K$, and
let $S_L$ and $S_M$ denote the respective
normalizations of $S$ in $L/K$ and $M/K$.
We suppose that $S_L$ and $S_M$ are \'etale over $S$.
Then the Blunk construction gives rise to a bijection between:
\begin{itemize}
\item Equivalence classes of smooth DP6 bundles
\[ \pi\colon X\to S \]
rigidified by $L/K$ and $M/K$, and
\item pairs of Brauer group elements $\eta\in \Br(S_L)[3]$ and
$\zeta\in \Br(S_M)[2]$ which restrict to zero in
$\Br(L\otimes_K M)$ and corestrict to zero in $\Br(K)$.
\end{itemize}
\end{theo}

\begin{proof}
A smooth DP6 bundle ridigified by $L/K$ and $M/K$
determines Brauer-Severi schemes
of relative dimension $2$ over $S_L$ (by birationally contracting a triple of
pairwise disjoint exceptional curves) and
of relative dimension $1$ over $S_M$ (by a corresponding
fibration in conics).
These have classes in $\Br(S_L)[3]$ and $\Br(S_M)[2]$, respectively;
since their restrictions to the generic points of all the components are the ones
that correspond to $X_K$ under Blunk's correspondence,
they satisfy the indicated conditions on restriction and corestriction.
Restriction to the generic point induces an
injective map on the Brauer group \cite[II.1.10]{GB}, hence
the Brauer group elements are uniquely determined by the
equivalence class of a given
rigidified smooth DP6 bundle.

It remains to show that any pair of Brauer group elements arises from some
smooth DP6 bundle.
We first treat the case $S=\Spec(R)$, where $R$ is a
semi-local Dedekind domain.
Then, we claim that
$\eta$ and $\zeta$ determine an element of
$H^1(\Spec(R),T)$, where $T$ denotes the two-dimensional torus associated
with $S_L/S$ and $S_M/S$.
The argument in \cite{blunk} carries over: in addition to basic functoriality
of the Brauer group one needs only the
concrete description of $H^1$ and $H^2$ of a norm one torus given at the
top of page 47 of op.\ cit.
Following the diagram in Figure \ref{figure1} and using the fact that
a semi-local Dedekind domain has trivial Picard group,
we see that the description remains valid.

Next we treat the general case.
The field case of Blunk's construction yields a
sextic del Pezzo surface over $K$, which spreads out to
a smooth DP6 bundle
\[ X'\to U \]
over a Zariski open dense subscheme $U\subset S$.
The complement $S\setminus U$ has finitely many
irreducible components of codimension $1$, whose generic points
we denote by $x_1$, $\dots$, $x_m$.
By \cite[Prop.\ VIII.1]{raynaudamples}, affine open subsets are
cofinal among all
open subsets $V$ of $\Sigma$ containing $x_1$, $\dots$, $x_m$.
Hence
\[ \varprojlim_{\substack{V\subset S\\ x_1,\dots,x_m\in V}} V \]
is an affine scheme of the form $\Spec(R_1\times\dots\times R_n)$ where each $R_i$
is a semi-local Dedekind domain.
By the case already treated, there is, for some $V$ as above, a
smooth DP6 bundle
\[ X''\to V \]
where, shrinking $V$ if necessary while maintaining
$x_1$, $\dots$, $x_m\in V$, we may suppose
\[ X'\times_U(U\cap V)\cong X''\times_V(U\cap V). \]
Then it is possible to glue $X'$ and $X''$ to obtain a smooth DP6 bundle
over $S^\circ:=S\setminus Z$ for some closed $Z$ that,
at each of its points, has codimension at least $2$.

We conclude by showing that restriction to
$S^\circ$ induces an equivalence of categories between
smooth DP6 bundles over $S$ and over $S^\circ$.
A smooth DP6 bundle $\pi^\circ \colon X^\circ \to S^\circ$
extends canonically to a projective scheme over $S$ as follows:
$X^\circ$ embeds in $\PP((\pi^\circ_*\omega_{X^\circ/S^\circ}^\vee)^\vee)$;
we apply direct image by $\iota\colon S^\circ\to S$ and closure to
obtain $X$ in $\PP(\iota_*(\pi^\circ_*\omega_{X^\circ /S^\circ}^\vee)^\vee)$.
These operations are compatible with \'etale base change.
So, it suffices to show, for $z\in Z$, with the
strictly henselian local ring of $S$ at $z$ denoted by
$\cO_{S,\bar z}$, that
$X^\circ \times_S\Spec(\cO_{S,\bar z})$ extends to a
smooth DP6 bundle over
$\cO_{S,\bar z}$.
This is clear, since
$T$ becomes a split torus after base change to $\cO_{S,\bar z}$.
\end{proof}

\begin{figure}
\begin{tiny}
\[
\xymatrix@C=8pt@R=10pt{
\frac{\G_m(S)}{N(\G_m(S_L))}\times
\frac{\G_m(S)}{N(\G_m(S_M))} \ar[r] \ar[d] &
H^1(S,R^1_L)\times
H^1(S,R^1_M) \ar[r] \ar[d] &
\ker(\Nm_L)\times\ker(\Nm_M) \ar[d] \\
\frac{\G_m(S)}{N(\G_m(S_{LM}))} \ar[r] &
H^1(S,R^1_{LM}) \ar[r] \ar[d] &
\ker(\Nm_{LM}) \\
&H^1(S,T) \ar[d] \\
\coker(\Nm_L)\times \coker(\Nm_M) \ar[r] \ar[d] &
H^2(S,R^1_L)\times
H^2(S,R^1_M) \ar[r] \ar[d] &
\ker(\cores_L)\times \ker(\cores_M) \ar[d] \\
\coker(\Nm_{LM}) \ar[r] &
H^2(S,R^1_{LM}) \ar[r] &
\ker(\cores_{LM})
}
\]
\end{tiny}
\caption{Exact sequence involving $H^1(S,T)$, from short exact
$0\to R^1_L\times R^1_M\to R^1_{LM}\to T\to 0$ of \cite{blunk}.
Here, $S_{LM}$ denotes $S_L\times_SS_M$ and
$R^1$ stands for the kernel of the norm $N$ from the restriction of scalars
$R(\G_m)$
under the indicated extension (e.g., $S_L$) of $S$
to $\G_m$.
Rows are short exact sequences (with leading and trailing $0$ omitted to
save space),
$\Nm$ denotes the norm map on Picard groups
(for the indicated extension, e.g.,
$\Nm_L\colon \Pic(S_L)\to \Pic(S)$), and
$\cores$ denotes corestriction of Brauer groups.}
\label{figure1}
\end{figure}

\begin{rema}
\label{rem.notuniqueisomorphismtype}
It is impossible to strengthen Theorem \ref{thm.blunksmooth} to a
uniqueness statement for the isomorphism class
of a rigidified smooth DP6 bundle,
even if we replace the Brauer classes by Azumaya algebra representatives
(or, what amounts to the same, Brauer-Severi schemes).
This is in constrast to the main theorem of \cite{APS}, which proves such a
result for quadric surface bundles.
Indeed, let $k$ be an algebraically closed field and
let $K$, $L$, and $M$ be rational function fields over $k$ of
transcendence degree $1$, such that $L\otimes_KM$ is a function field of
positive genus.
Take $S$ to be the complement of a suitable finite set of
points in $\PP^1_k$.
The hypotheses of Theorem \ref{thm.blunksmooth} are
satisfied, and any smooth DP6 bundle over $S$
rigidified by $L/K$ and $M/K$ has
associated Brauer-Severi schemes $\PP^2_{S_L}$
and $\PP^1_{S_M}$ (by Tsen's theorem and the fact that the
coordinate rings of $S_L$ and $S_M$ are PIDs).
However, we have nontrivial
$H^1(S,T)\cong \Pic(S_L\times_SS_M)$
(cf.\ Figure \ref{figure1}).
Thus, there is more than one
isomorphism class of smooth DP6 bundles over $S$
rigidified by $L/K$ and $M/K$.
\end{rema}

\section{Basic degenerations}
\label{sect.basicdegen}
In this section we analyze the five degeneration types in
DP6 bundles with basic degeneration.
At the generic point of a divisor on the base
they correspond to the possible degeneration types of the
rigidification data.

Let $\mathfrak{o}_K$ be a DVR with fraction field $K$ and
residue field $\kappa$ of characteristic different from $2$ and $3$.
Let $L/K$ and $M/K$ be rigidification data.
Let $\mathfrak{o}_L$, respectively $\mathfrak{o}_M$ denote the
integral closure of $\mathfrak{o}_K$ in $L$, respectively in $M$.

There are the following possibilities for $L$:
\begin{itemize}
\item[$(L_u)$] $L/K$ is unramified.
\item[$(L_r)$] $L/K$ is ramified.
\end{itemize}
The following are the possibilities for $M$:
\begin{itemize}
\item[$(M_u)$] $M/K$ is unramified.
\item[$(M_r)$] $M/K$ is simply ramified.
\item[$(M_t)$] $M/K$ is totally ramified.
\end{itemize}
The possible combinations are listed in Table \ref{tableLM}.

\begin{table}
\[
\begin{array}{c|cc}
& (L_u) & (L_r) \\ \hline
(M_u) & \text{smooth}   &  \text{I} \\
(M_r) & \text{II} &  \text{III} \\
(M_t)  & \text{IV}   &  \text{V}
\end{array}
\]
\caption{Basic degenerations}
\label{tableLM}
\end{table}

We describe, for each type, the root constructions that replace
ramified covers of schemes by unramified covers of stacks.
We work over $S=\Spec(\mathfrak o_K)$,
and introduce $S_L=\Spec(\mathfrak{o}_L)$ and $S_M=\Spec(\mathfrak{o}_M)$.
\begin{itemize}
\item {\it Type I:} Replace $S$ by $\sqrt{(S,\Spec(\kappa))}$,
make a corresponding replacement of $S_M$.

\item {\it Type II:} Replace $S$ by $\sqrt{(S,\Spec(\kappa))}$,
make a corresponding replacement of $S_L$,
replace $S_M$ by its root stack along the
unramified closed point.

\item {\it Type III:} Replace $S$ by $\sqrt{(S,\Spec(\kappa))}$,
replace $S_M$ by its root stack along the
unramified closed point.

\item {\it Type IV:} Replace $S$ by $\sqrt[3]{(S,\Spec(\kappa))}$,
make a corresponding replacement of $S_L$.

\item {\it Type V:} Replace $S$ by $\sqrt[6]{(S,\Spec(\kappa))}$,
perform cube- respectively square-root replacements of
$S_L$, respectively $S_M$.
\end{itemize}

In summary, we replace $S$ by $\sqrt[n]{(S,\Spec(\kappa))}$,
where $n=2$ for Types I, II, and III;
$n=3$ for Type IV; and $n=6$ for Type V.
For the corresponding replacements $S'_L$ of $S_L$ and
$S'_M$ of $S_M$ we have achieved the following:

\begin{prop}
\label{prop.replace}
In every case, the morphisms
\[ S'_L\to \sqrt[n]{(S,\Spec(\kappa))}\qquad\text{and}\qquad
S'_M\to \sqrt[n]{(S,\Spec(\kappa))} \]
are finite and \'etale.
\end{prop}

The classification of basic degeneration types reflects the classification of
conjugacy classes of cyclic subgroups of $\mathfrak S_2\times \mathfrak S_3$,
acting on toric $X_0$ with $T\subset X_0$, and in particular, on the set of exceptional curves;
see \eqref{eqn:aut}.

\begin{itemize}
\item {\it Type I:} The factor $\mathfrak S_2$, which acts
by the inverse map on $T$
and by exchanging opposite pairs of exceptional curves.
\item {\it Type II:} Cyclic subgroup of order $2$ contained in the factor
$\mathfrak S_3$, which acts, fixing a smooth conic in $X_0$.
\item {\it Type III:} The remaining order $2$ case, where the fixed locus
is a smooth quartic curve in $X_0$ passing through a pair of opposite singular
points of the exceptional locus.
\item {\it Type IV:} Cyclic subgroup of order $3$, acting with three fixed points on $T$ and with two orbits on the set of exceptional curves.
\item {\it Type V:} Cyclic subgroup of order $6$, acting with one fixed point in $T$,
one orbit of three points with stabilizer $\mu_2$, and
one orbit of two points with stabilizer $\mu_3$.
\end{itemize}

\section{Construction}
\label{sect.construction}
In this section, we describe the construction of basic degenerations. The construction proceeds in three steps:
\begin{itemize}
\item[(1)] Root stack construction on the base to permit a DP6 bundle to
extend smoothly across a given divisor.
(See Section \ref{sect.basicdegen}.)
\item[(2)] Birational modification of the smooth DP6 bundle on the root stack.
This step combines the operations of
blowing up, contracting \cite[Prop.\ A.9]{KTsurf}, and
pinching \cite{ferrand}.
\item[(3)] Descent from the root stack to the original base
\cite[Prop.\ 2.5]{KTsurf}.
\end{itemize} 

Step (1), when the base is $S=\Spec(\mathfrak o_K)$, replaces $S$ by
\[ \sqrt[n]{(S,\Spec(\kappa))}, \]
where $n=2$ for Types I, II, and III;
$n=3$ for Type IV; and $n=6$ for Type V.
In the case of a semi-local Dedekind domain, we perform the
corresponding root stack replacement at each closed point.
The passage from semi-local Dedekind domain case to general case will proceed just as
in the proof of Theorem \ref{thm.blunksmooth}.

Step (2) breaks up according to the Type.
We let $s$ denote the gerbe of the root stack;
$X_s$ is a smooth DP6 with $\mu_n$-action.

\begin{itemize}
\item {\it Type I:}
The fixed-point locus consists, geometrically, of four points in $X_s$, which we
blow up.
The special fiber becomes a union of four copies of $\PP^2$, attached along lines to
four $(-1)$-curves on a resolved DP2, double cover of the plane branched along
the union of four lines. The resolved DP2 has six $(-2)$-curves, which may be
flopped to obtain the singular DP2 with four attached copies of the
blow-up of a plane at three collinear points.
The DP2 contracts, leaving four copies of the
plane blown up at three collinear points with $(-2)$-curve contracted,
attached to each other along exceptional divisors.
This is $\Q$-Gorenstein of index $2$.
Twice the anticanonical class gives rise to an embedding in
$\PP^{18}$,
where the image of each component has degree $6$.

\item {\it Type II:}
Let $C\subset X_s$ be the fixed-point locus, a smooth conic
incident to two opposite exceptional curves.
The blow-up $B\ell_CX$ has, as fiber over $s$, the union of $X_s$ and
$C\times \PP^1$, and
the two incident exceptional curves on $X_s$ may be
flopped to obtain a $\mathrm{DP8}$ joined along $C$ with
the blow-up of $C\times \PP^1$ at two points on $C$.
(Note that the exceptional curves, and the points, may be Galois conjugated.)
The $\mathrm{DP8}$ may be contracted to an
ordinary double point singularity on a new
flat family of sextic del Pezzo surfaces over $S$ with Type II
fiber over $s$.

\item {\it Type III:}
Let $C\subset X_s$ be the fixed-point locus,
a quartic rational curve $C$ through
two opposite singular points of the exceptional locus.
The blow-up $B\ell_CX$ has, as fiber over $s$, the union of $X_s$ and a
Hirzebruch surface $\F_2$, where $C$ is identified with the $(-2)$-curve of $\F_2$.
There is a conic bundle $X_s\to \PP^1$, which restricts to
a degree $2$ morphism $C\to \PP^1$.
A corresponding contraction of $B\ell_CX$ has the effect, on the
special fiber, of ``pinching'' $\F_2$ onto its image under the
incomplete linear system of four times a ruling plus
the $(-2)$-curve, consisting of sections whose restriction
to $C$ are pullbacks of sections of $\cO_{\PP^1}(1)$ under $C\to \PP^1$.

\item {\it Type IV:} The fixed-point locus consists, geometrically,
of three points in $X_s$, which we blow up.
The special fiber becomes three copies of $\PP^2$, attached along lines
$\ell$, $\ell'$, $\ell''$
to a smooth cubic surface with $18$ Eckhardt points
(and thus, geometrically, the Fermat cubic surface).
In each copy of $\PP^2$ the action of $\mu_3$ fixes a point and a line;
we blow these up, which replaces the cubic surface by
its blow-up at six of the Eckhardt points, upon which
the self-intersection number changes from $-1$ to $-3$ for
nine exceptional curves: $\ell$, $\ell'$, $\ell''$, and
six others.
The six others each have normal bundle $\cO_{\PP^1}(-3)\oplus \cO_{\PP^1}(-3)$,
and we flop these to obtain
curves, along which the total space
has singularities of type $\mathsf A_{3,1}$, cf.\ \cite[\S III.5]{BHPV},
also known as $\frac{1}{3}(1,1)$.
The remaining three are fibers of Hirzebruch surfaces $\F_1$
(blow-up of a point on $\PP^2$), which we
contract to lines.
After the flops and contractions
the cubic surface has become a surface with singularity type
$9\mathsf A_{3,1}$ and $K^2=0$; in fact, it is geometrically the quotient
of $E\times E$ by the diagonal $\mu_3\subset \mu_3\times \mu_3$,
where $E$ is an elliptic curve of $j$-invariant $0$
and $\mu_3$ acts by elliptic curve automorphisms
(cf., e.g., \cite{tokunagayoshihara}).
The components of the special fiber of multiplicity $2$ have pointwise trivial
but scheme-theoretically nontrivial action of $\mu_3$, making the
quotient map a finite flat morphism of degree $2$ to the
respective reduced subscheme.
The operation of pinching by this morphism
\cite{ferrand} introduces singularities along these components but
transforms the non-Gorenstein $\mathsf A_{3,1}$-singularities
into hypersurface singularities.
Then we contract the non-normal component of the special fiber
to a point.

\item {\it Type V:} On the special fiber
$\mathrm{DP6}$ there is a unique $\mu_6$-fixed point, which we blow up
to obtain a $\mathrm{DP5}$ joined along exceptional curve $\ell$ to
a copy of $\PP^2$ on which the action of $\mu_2\subset \mu_6$ is trivial
and the action of $\mu_3\subset \mu_6$ fixes a point and a line;
the line $m$ that is fixed meets $\ell$ at a point $q$.
On the $\mathrm{DP5}$ there are
three (possibly Galois conjugated) exceptional curves $\ell'$, $\ell''$, $\ell'''$ incident to $\ell$.
As well, the proper transforms of cubic curves on $\mathrm{DP6}$ through the
$\mu_6$-fixed point comprise two pencils of conics on $\mathrm{DP5}$,
out of which we are interested in the unique members passing through $q$.
We blow up $m$, which has normal bundle isomorphic to
$\cO_{\PP^1}(1)\oplus \cO_{\PP^1}(-1)$, so the exceptional divisor is
a copy of the Hirzebruch surface $\F_2$, and the action of
$\mu_2$ on $\F_2$ fixes a pair of sections, while the
action of $\mu_3$ is trivial.
The proper transforms of the two conics through $q$ are disjoint,
each with normal bundle $\cO_{\PP^1}(-1)\oplus \cO_{\PP^1}(-1)$; they may be
flopped, so that the Hirzebruch surface is blown up at two
points on a fiber.
We contract $\ell'$, $\ell''$, and $\ell'''$ to points, so that
the special fiber becomes the union of the blown up Hirzebruch surface and
two copies of $\PP^2$ meeting along a line on which the total space has,
geometrically, three ordinary double point singularities.
The union of the two copies of $\PP^2$ is a Cartier divisor, which
may be contracted to a threefold singularity of type $\mathsf D_4$.
Since, then, the action of $\mu_3$ on the special fiber is trivial,
we may descend so that on the base there is only $\mu_2$-stabilizer,
and the fiber is the $\mathsf A_2$-singular surface that is obtained
by contracting the pair of $(-2)$-curves of the blown up
Hirzebruch surface; the $\mu_2$-fixed locus is a rational curve $n$.
Blowing up $n$, which has normal bundle isomorphic to
$\cO_{\PP^1}(3)\oplus \cO_{\PP^1}(-1)$,
produces exceptional divisor $\F_4$ and replaces the
$\mathsf A_2$-singularity with an $\mathsf A_1$-singularity.
The $\mathsf A_1$-singular surface may be contracted to a curve, leaving
$\F_4$ with $(-4)$-curve glued to a fiber.
\end{itemize}

Step (3) is straightforward in Types I through IV.
In Type V there is an intermediate descent step which replaces the
original $\mu_6$-action by a $\mu_2$-action;
a similar two-stage descent construction has been employed in
the proof of \cite[Thm.\ 6]{KTinvsurf}.

\begin{theo}
\label{thm.construction}
Let $S$ be a quasi-compact separated regular scheme, in whose local rings $2$ and $3$ are invertible, and
let $D_1$, $\dots$, $D_5$ be disjoint regular divisors.
Then the construction described above identifies, up to unique isomorphism:
\begin{itemize}
\item smooth $\mathrm{DP6}$ bundles $\mathcal{X}\to \mathcal{S}$, where $\mathcal{S}$ denotes the
iterated root stack
\begin{equation}
\label{eqn.iteratedrootstack}
\sqrt{(S,D_1\cup D_2\cup D_3)}\times_S\sqrt[3]{(S,D_4)}\times_S\sqrt[6]{(S,D_5)},
\end{equation}
such that on geometric fibers over the gerbe of the root stack over $D_i$ the action of
$\mu_2$ ($i\le 3$), $\mu_3$ ($i=4$), $\mu_6$ ($i=5$) is a toric action of Type $i$.
\item basic $\mathrm{DP6}$ bundles $X\to S$ with
Type $i$ fibers over $D_i$ for $i=1$, $\dots$, $5$,
where $X$ is regular, except for
isolated singularities of type cone over Veronese surface over $D_1$ and
of type cone over rational normal scroll $S_{2,2}$ over $D_4$.
\end{itemize}
\end{theo}

\begin{proof}
The construction described above is checked, in each Type,
to transform a smooth $\mathrm{DP6}$ bundle over the root stack
to a basic $\mathrm{DP6}$ bundle.
The construction may be reversed.
The pattern of argument follows \cite[\S 5--\S 6]{KTsurf}, except that
for the reverse construction there is a new ingredient: the pinching operation in
Type IV is reversed by normalization.

We give extensive details in Type III and sketch the arguments
in Types IV and V.

\medskip

\emph{Type III:} The construction evidently leads
to a $\mathrm{DP6}$ bundle with Type III fibers over $D_3$.
It remains to verify that the total space is regular.
For this we may work locally, and assume we are at a point of $D_3$ with
separably closed residue field, $D_3$ defined by $f=0$, and root stack with
gerbe of the root stack defined by $t=0$ where $t^2=f$.
Write $\mathrm{DP6}$ has the hypersurface $xyz=\tilde x\tilde y\tilde z$ in
$\PP^1\times \PP^1\times \PP^1$, with homogeneous coordinates
$x$, $\tilde x$, etc., on the respective $\PP^1$ factors, where $\mu_2$ acts by
swapping the first two factors.
The conic bundle is given by projection to the third factor,
and $C$ is defined by $x\tilde y=\tilde xy$ in the $\mathrm{DP6}$,
passing through $(0,0,\infty)$ and $(\infty,\infty,0)$.
We denote the exceptional curves on the $\mathrm{DP6}$ by
$E_{0 \infty \PP}$, etc., reflecting the images under projection to the
three $\PP^1$ factors.
Then we compute that multiplication by
$t/(x\tilde x^{-1}-y\tilde y^{-1})$ identifies
\[
\cO(E_{0 \infty \PP} + E_{\PP \infty 0} + E_{\infty \PP 0} + E_{\infty 0 \PP})
\]
with
$\cI/\cI^2$, where $\cI$ denotes the ideal sheaf of the family of DP6 over the
gerbe of the root stack, which is to be contracted.
By \cite[Prop.\ A.9]{KTsurf}, contraction yields
$\PP^1$ over the gerbe of the root stack whose normal cone is identified with
$\Spec(\bigoplus_{n\ge 0} \psi_*(\cI/\cI^2)^n)$, where $\psi$ denotes the
contraction to $\PP^1$.
We compute:
\[ \psi_*(\cO(E_{0 \infty \PP} + E_{\PP \infty 0} + E_{\infty \PP 0} + E_{\infty 0 \PP}))\cong \cO\oplus \cO\oplus \cO(1), \]
where global sections  $x\tilde x^{-1}$, $y\tilde y^{-1}$, $1$, $x\tilde x^{-1}y\tilde y^{-1}$,
on the left correspond to $(1,0,0)$, $(0,1,0)$, $(0,0,z)$, $(0,0,\tilde z)$, respectively,
on the right; the action of $\mu_2$ swaps the first two factors.
The anti-invariant section $t$ corresponds to $(1,-1,0)$.
There is a quadratic relation with term $t^2$,
defining a hypersurface singularity of type $\mathsf D_\infty$.
(Such singularities have been encountered in \cite[\S 3]{KTinvsurf}.)
After descent this yields a hypersurface whose
defining equation includes a nontrivial linear term $f$, hence
the total space is regular.

\medskip

\emph{Type IV:} We indicate, in coordinates, the curves in the
blown-up $\mathrm{DP6}$ that are flopped.
Write the $\mathrm{DP6}$ in the standard way as
compactification of $\G_m^2$ with coordinates $x$ and $y$,
with action of primitive third root of unity $\omega$ by
\[ (x,y)\mapsto (x^{-1}y,x^{-1}). \]
Then, $\ell$, $\ell'$, $\ell''$ are the exceptional curves obtained by blowing up
\[ p:=(1,1),\qquad p':=(\omega,\omega^2),\qquad p'':=(\omega^2,\omega). \]
The exceptional curves on the cubic surface are listed in
Table \ref{tableexceptional}.
The cubic surface is blown up at two points on each of
$\ell$, $\ell'$, $\ell''$, namely the Eckhardt points, incident to the
proper transforms of the cubic curves in Table \ref{tableexceptional}.
Those become the six $(-3)$-curves that are flopped.
Then, the blown-up $\mathrm{DP6}$ is contracted to a surface with
singularity type $6\mathsf A_{3,1}$, with three
floating $(-1)$ curves (i.e., contained in the smooth locus)
corresponding to the three quartic curves in Table \ref{tableexceptional}.
If the floating $(-1)$ curves would be contracted, the surface would have
anticanonical degree $2$ and would be the surface
listed in \cite{miura} as No.\ 9 in \cite[Table 1]{miura}.
In fact, this is a toric surface, with fan
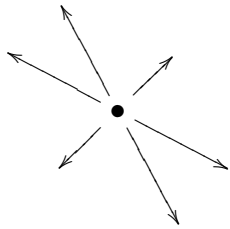
\[
\xymatrix@R=14pt@C=14pt{
&\\
&&&\\
&&{\bullet}\ar[ull]\ar[uul]\ar[ur]\ar[drr]\ar[ddr]\ar[dl]\\
&&&&\\
&&&
}
\]
and three points,
incidence points to fibers of each of the three visible conic bundle structures,
whose blow-ups recover the three floating $(-1)$-curves.
The three remaining $(-3)$-curves correspond to
$\ell$, $\ell'$, $\ell''$.
When contracted, the surface acquires anticanonical degree $0$,
and is transformed by the pinching operation to the degree $3$ cover
of $\PP^1\times \PP^1$ of the form $\Spec(\mathcal{A})$, where $\mathcal{A}$ is
a coherent sheaf of algebras of the form
\[ \cO_{\PP^1\times \PP^1}\oplus
\cO_{\PP^1\times \PP^1}(-1,-2)\oplus
\cO_{\PP^1\times \PP^1}(-2,-4), \]
with algebra structure determined by a morphism
\[ \cO_{\PP^1\times \PP^1}(-3,-6)\to \cO_{\PP^1\times \PP^1} \]
vanishing along the union of three rulings with multiplicity $1$ and
three rulings (from the other family of rulings) with multiplicity $2$.
This surface is singular along the multiplicity $2$ rulings.
Although the structure sheaf of the surface has nonvanishing $H^2$,
the conclusion of \cite[Prop.\ A.9]{KTsurf} still holds in the strong form
needed here (flatness, compatibility with base change).
Indeed, if $\psi\colon X\to Y$ denotes the contraction,
$X\cong B\ell_FY$
(where $F\subset Y$ is a section of the restriction of $Y$ to the
gerbe of the root stack), with line bundle $L\cong \psi^*\cO_Y(1)$ on $X$,
then after replacing $L$ by a suitable power we find that
$R^1\psi_*(L^n)=0$ and
$R^2\psi_*(L^n)$ is (the direct image under $F\to Y$ of)
a locally free sheaf on $F$, for all $n\ge 1$.
The latter has projective dimension $1$
at points of $F$.
By the machinery of cohomology and base change,
the formation of $\psi_*(L^n)$ commutes with arbitrary base change,
i.e., the conclusion of \cite[Prop.\ A.8(iv)]{KTsurf} holds.
\begin{table}
\[
\begin{array}{c|c|c}
\text{number}&\text{description}&\text{defining equations} \\ \hline
3&\text{exc.\ of blow-up}& \\ \hline
6&\text{proper transf.\ of exc.\ curves}& \\ \hline
9&\text{proper transf.\ of conics}&x\in \mu_3,\,\,y\in \mu_3,\,\,y/x\in \mu_3 \\ \hline
&& {\scriptstyle x+y+1=0,\,\, x+y+xy=0,} \\
6&\text{proper transf.\ of cubics}& {\scriptstyle \omega x+\omega^2y+1=0,\,\, \omega x+\omega^2y+xy=0,} \\
&& {\scriptstyle \omega^2x+\omega y+1=0,\,\, \omega^2x+\omega y+xy=0} \\ \hline
3&\text{proper transf.\ of quartics}&
xy=1,\,\, y=x^2,\,\, x=y^2
\end{array}
\]
\caption{The $27$ exceptional curves on the blow-up of $\mathrm{DP6}$ at
three points}
\label{tableexceptional}
\end{table}

\medskip
\emph{Type V:} The step which contracts a conic bundle over a
rational curve with $\mathsf A_1$-singular total space
yields a singularity of type $\mathsf J_{2,\infty}$
\cite{goryunov}, \cite{siersma}.
\end{proof}

\begin{theo}
\label{thm.blunknotsosmooth}
Let $S$ be a quasi-compact separated regular scheme, in whose local rings $2$ and $3$ are invertible, with
rigidification data $L/K$, $M/K$, and
let $D_1$, $\dots$, $D_5$ be disjoint regular divisors.
We let $S_L$ and $S_M$ denote the respective normalizations of $S$ in $L/K$ and $M/K$.
We suppose that $S_L\to S$ is finite and flat, ramified over $D_1\cup D_3\cup D_5$,
and $S_M\to S$ is finite and flat, simply ramified over $D_2\cup D_3$ and
totally ramified over $D_4\cup D_5$.
We introduce notation $D_{L,i}$ and $D_{M,i}$ for the reduced divisor
of $S_L$, respectively $S_M$, over $D_i$, for $i=1$, $\dots$, $5$,
and write
\[ D_{M,2}=D_{M_u,2}\cup D_{M_r,2},
\qquad
D_{M,3}=D_{M_u,3}\cup D_{M_r,3}, \]
to distinguish the components where the covering map is unramified and
simply ramified.
Then the Blunk construction gives rise to a bijection between
\begin{itemize}
\item Equivalence classes of basic $\mathrm{DP6}$ bundles with
Type $i$ fibers over $D_i$ for $i=1$, $\dots$, $5$, where the total space
is regular, exception for
isolated singularities of type cone over Veronese surface over $D_1$ and
of type cone over rational normal scroll $S_{2,2}$ over $D_4$,
rigidified by $L/K$ and $M/K$, and
\item pairs of Brauer group elements
\[ \eta\in \Br(S_L)[3] \setminus (D_{L,4} \cup D_{L,5}) \]
and
\[ \zeta\in \Br(S_M)[2] \setminus (D_{M,1} \cup D_{M_u,2} \cup D_{M_u,3} \cup
D_{M,5}) \]
which restrict to zero in
$\Br(L\otimes_K M)$ and corestrict to zero in $\Br(K)$.
\end{itemize}
\end{theo}

Based on Proposition \ref{prop.replace}, we see that the
iterated root stack \eqref{eqn.iteratedrootstack}
has finite \'etale covers
\begin{equation}
\label{eqn.deg2cover}
\sqrt{(S_L,D_{L,2})}\times_{S_L}\sqrt[3]{(S_L,D_{L,4}\cup D_{L,5})}
\end{equation}
of degree $2$ and
\begin{equation}
\label{eqn.deg3cover}
\sqrt{(S_M,D_{M,1}\cup D_{M_u,2}\cup D_{M_u,3}\cup D_{M,5})}
\end{equation}
of degree $3$.

\begin{proof}
The passage from basic $\mathrm{DP6}$ bundles to Brauer group elements is
just as in the first paragraph of the proof of Theorem \ref{thm.blunksmooth}.
The rest of the proof is structured just as in the remainder of the
proof of Theorem \ref{thm.blunksmooth}, where we note that by
Theorem \ref{thm.construction}, it suffices to exhibit a smooth
$\mathrm{DP6}$-fibration over the iterated root stack \eqref{eqn.iteratedrootstack}.
The argument
in the case of a semi-local Dedekind domain relies
on the fact that the given Brauer group elements extend to
the Brauer groups of \eqref{eqn.deg2cover}, respectively
\eqref{eqn.deg3cover},
since the orbifold structure kills any ramification, as observed, e.g., in
\cite[Lemma 2]{oesinghaus},
and on the fact that in
Figure \ref{figure1}, the top-right vertical map is surjective,
while the bottom-left vertical map is an isomorphism.
Indeed, the target of the top-right vertical map is zero in
Types I, III, IV, and V, and in Type II this map is an
isomorphism $\Z/2\Z\to \Z/2\Z$.
Case-by-case verification takes care of the claim concerning the
bottom-left vertical map, e.g.,
in Type II this is $0\to 0$ when $L/K$ is split, and otherwise
$\Nm_L$ and $\Nm_{LM}$ are both zero, while $\Nm_M$ is surjective.
The remainder of the argument is exactly as in the proof of
Theorem \ref{thm.blunksmooth}: restriction to the
complement of a closed substack of codimension at least $2$ induces
an equivalence of categories between smooth
$\mathrm{DP6}$ bundles over the iterated root stack
\eqref{eqn.iteratedrootstack} and
smooth $\mathrm{DP6}$ bundles over the complement.
\end{proof}

\bibliographystyle{plain}
\bibliography{brauer}

\end{document}